\documentclass[a4paper, 12pt, reqno]{amsart}
\usepackage{tikz}
\usepackage[utf8]{inputenc}
\usepackage[T1]{fontenc}
\usepackage{indentfirst}
\usepackage{amsmath} 
\usepackage{amsfonts} 
\usepackage{amsthm} 
\usepackage{amscd}
\usepackage{color}
\usepackage{graphics}
\usepackage{icomma} 
\usepackage{amsmath} 
\usepackage{amsfonts} 
\usepackage{enumerate}
\usepackage{dsfont,  amssymb, stmaryrd, latexsym, mathrsfs}
\usepackage{xypic}
\usepackage{array}
\usepackage{longtable}
\usepackage{geometry}
\usepackage[frenchb,english]{babel}
\usepackage{color}
\usepackage{csquotes}
\usepackage{graphics}
\usepackage{multirow} 
\usepackage{multicol}
\theoremstyle{plain}
\newtheorem{them}{Theorem}[section]
\newtheorem{lem}{Lemma}[section]

\theoremstyle{remark}
\newtheorem{rema}{\bf Remark}[section]

\newtheorem{exams}{\textbf{Numerical Examples}}[section]

\def\NN{\mathds{N}}

\def\QQ{\mathbb{Q}}

\begin{document}
\selectlanguage{english}

\title[The $2$-rank of the class...]{The 2-rank of the class group of some real cyclic quartic number fields}
\author[A. Azizi]{Abdelmalek Azizi}
\address{Abdelmalek Azizi: Mohammed First University, Sciences Faculty, Mathematics Department, Oujda, Morocco}
\email{abdelmalekazizi@yahoo.fr}
\author[M. Tamimi]{Mohammed Tamimi}
\address{Mohammed Tamimi: Mohammed First University, Sciences Faculty, Mathematics Department, Oujda, Morocco}
\email{med.tamimi@gmail.com}
\author[A. Zekhnini]{Abdelkader Zekhnini}
\address{Abdelkader Zekhnini: Mohammed First University, Pluridisciplinary Faculty, Mathematics Department, Nador, Morocco}
\email{zekha1@yahoo.fr}

\subjclass[2000]{11R16; 11R29; 11R11; 11R80.}
\keywords{Real cyclic quartic number field, $2$-rank, $2$-class group, quadratic fields}

\maketitle

\selectlanguage{english}

\begin{abstract}
In this paper, we investigate  the $2$-rank of the class group of some real cyclic quartic number fields. Precisely, we consider the case where the quadratic subfield is $\mathbb{Q}(\sqrt{\ell})$ with $\ell\equiv5\,\pmod8$ is a prime.
\end{abstract}
\section{\bf Introduction}\label{sec1}
Let $K$ be a number field and $H$ its  $2$-class group, that is the $2$-Sylow subgroup of the ideal class group of $K$. The rank $r_2(H)$ of $H$  is the number of cyclic $2$-groups appearing in the decomposition of $H$ whose orders are a $2$ power,  that is   the dimension of the $\mathbb{F}_2$-vector space $H/H^2$, where $\mathbb{F}_2$ is the field of $2$ elements.

Many mathematicians  are interested in determining $r_2(H)$ and the power of $2$ dividing the class number of $K$.  Hasse \cite{Ha70}, Bauer \cite{Ba71} and others gave methods to determine the exact power of $2$ dividing the class number of a quadratic number field. These methods were developed by C. J. Parry and his co-authors  to determine $r_2(H)$ and the  power of $2$ dividing the class number of some  imaginary cyclic quartic  number field $K$ having a quadratic subfield $k$ with odd class number (e.g., \cite{Parry 77, BrPa78, Parry 75, Pa75, Pa80}). For this,  they needed a suitable genus theory convenient to their situation. Hence they showed that  the theory firstly developed by Hilbert (\cite{Hi94}), assuming an imaginary base field $k$, can be adapted to the situation where $K$ is a totally imaginary quartic cyclic extension of a totally real quadratic subfield  $k$. This theory can be applied, with minor modifications,  to any quartic number field $K$ having a quadratic subfield $k$ of odd class number.

The $2$-rank of any biquadratic number field  $K$  is determined (partially or totally) in many paper   (\cite{McPaRa95, McPaRa97, Parry 77, BrPa78, AM01, AM04}) up to the case: $K$ is a real quartic cyclic extension of the rational number field $\mathbb{Q}$. Denote by $k=\mathbb{Q}(\sqrt\ell)$ its unique quadratic subfield, we aim  to investigate this case,  whenever $k$ has odd class number and the norm of its fundamental unit equals $-1$, i.e. $\ell=2$ or $\ell$ is a prime congruent to $1\pmod4$. In this paper, we restrict ourselves to the case where  $\ell$ is a prime congruent to $5\pmod8$ for two reasons: to avoid long paper and the technics used in this case  are little bit different from those used in the other cases.

An outline of the paper is as follows. In  \S\ \ref{2} we summarize some preliminary results on quartic cyclic number fields  and the ambiguous class number formula, which we will use later.   The main theorems   are presented and proved in \S\ \ref{15}. In \S\ \ref{16} we characterize all real cyclic quartic number fields $K=\mathbb{Q}(\sqrt{n\epsilon_{0}\sqrt{\ell}})$, with $l$ is a prime congruent to $5\pmod8$,  having a $2$-class group trivial, cyclic, of rank $2$ or of rank $3$.
\section*{\bf Notations}
\noindent Throughout this paper, we adopt the following notations.
\begin{enumerate}[$\bullet$]
  \item $\mathbb{Q}$: the rational field.
  \item $\ell$: a prime integer congruent  to $5$ modulo $8$.
  \item $k=\mathbb{Q}(\sqrt{\ell})$: a quadratic field.
  \item $\epsilon_{0}$: the fundamental unit of $k$.
  \item $n$: a square-free positive integer relatively prime to $\ell$.
  \item $\delta=1$ or $2$.
  \item $d= n\epsilon_{0}\sqrt{\ell}$.
  \item $\mathbb{K}=k(\sqrt{d})$: a real quartic cyclic number field.
  \item $\mathcal{O}_k$ (resp. $\mathcal{O}_\mathbb{K}$): the ring of integers of $k$ (resp. $\mathbb{K}$).
  \item $H$: the 2-class group of  $\mathbb{K}$.
  \item $\mathfrak{2}$: the prime ideal of $k$ above $2$.
  \item $r_2(H)$: the rank of $H$.
  \item $\mathbb{K^{*}}$, $k^*$: the nonzero elements of the fields $\mathbb{K }$ and $k$ respectively.
  \item $N_{\mathbb{K}/k}(\mathbb{K})$: elements of $k$ which are norm from $\mathbb{K}$.
\item $p,\; q,\; p_{i},\; q_{j}$: odd prime integers.
\item $(\frac{x,\,y}{p})_{k}$: quadratic norm residue symbol over $k$.
\item $[\frac{\alpha}{\beta}]$: quadratic residue symbol for $k$.
\item $(\frac{a}{b})$: quadratic residue (Legendre) symbol.
\item $(\frac{a}{b})_{_{4}}$: rational $4$-th power residue symbol.
\end{enumerate}

\section{\bf Preliminary results}\label{2}
Let $K$ be a cyclic quartic extension of the rational number field $\mathbb{Q}$. By  \cite[Theorem 1]{wi87}, it is known  that $K$ can be expressed uniquely in the form $K=\mathbb{Q}(\sqrt{a(\ell+b\sqrt{\ell})})$, where $a, b, c$ and $\ell$ are integers satisfying the conditions: $a$ is odd and square-free, $\ell= b^2+c^2$ is  square-free, positive and relatively prime to $a$, with $b>0$ and $c>0$. Note that $K$  possesses a unique quadratic subfield $k=\mathbb{Q}(\sqrt\ell)$. Assuming  the class number of $k$  odd and $N_{k/\mathbb{Q}}(\epsilon_0)=-1$, then, by \cite{Xianke 84}, this is equivalent to that $K/k$ has an integral basis, and by \cite{Parry 90}, this is equivalent to the existence of an integer $n$ such that $K=\mathbb{Q}(\sqrt{n\epsilon_0\sqrt{\ell}})$ with $\epsilon_{0}$ is the fundamental unit of  $k$,  and
		\begin{center}
			$n=
			\left \{
			\begin{array}{rl}
			2a & \text{ if }\ell\equiv1\pmod4\text{ and }b\equiv 1 \pmod2\\
			a & \text{ otherwise}.
			\end{array}
			\right.
			$\end{center}

 Recall that the extensions $K/\mathbb{Q}$ were  investigated  by  Hasse in a paper (\cite{Hasse 48})  prior to that of Leopoldt (\cite{Leo}) on  the arithmetic interpretation of the class number  of real abelian fields. They were also investigated  by M. N. Gras \cite[...]{Gras-M-N 80-81, Gr77, Gr79} and others. We have the following remark.
\begin{rema}
Keeping the notations above of a cyclic quartic field, $K$ is then real if and only if $a> 0$ (equivalently $n>0$).
\end{rema}
\noindent The field $K$ also satisfies the following lemma.
\begin{lem}[\cite{Zink 66-67}]\label{lem: Zink 66-67}
	Let  $a, b$ and $c$ be positive integers,  $\ell= b^{2} + c^{2}$, with $a$ and $c$  odd, then $\mathbb{Q}(\sqrt{2a(\ell+b\sqrt{\ell}}))=\mathbb{Q}(\sqrt{a(\ell+c\sqrt{\ell}}))$.
\end{lem}
We will need the following theorem about the conductor  $f_K$ of $K$.
\begin{them}[\cite{wi87}]\label{wi87}
	The  conductor $f_K$  of the $($real or imaginary$)$ cyclic quartic field $K=\mathbb{Q}(\sqrt{a(\ell+b\sqrt{\ell}})$, where $a$ is an odd square-free integer, $\ell= b^2+c^2$ is a square-free positive integer relatively prime to $a$, with $b>0$ and $c>0$, is given by  $f_\mathbb{K}=2^{e}|a|\ell$, where   $e$  is defined by:
	$$e=
	\left \{
	\begin{array}{l}
	3,\; if\; \ell\equiv 2 \pmod8\;
	or\; \ell \equiv 1\pmod4 \;and\; b \equiv 1 \pmod2,\\
	2, \;if\; \ell \equiv 1\pmod4,\; b \equiv 0\, \pmod2,\; a+b \equiv 3\pmod4,\\
	0,\; if\; \ell \equiv 1\pmod4, \;b \equiv 0\, \pmod2,\; a+b \equiv 1\pmod4.
	\end{array}
	\right.
	$$		
\end{them}
We end this section by  recalling  the number of ambiguous ideal classes of a quadratic extension $K/k$.
\begin{them}[\cite{Azizi 04, Parry 75}]\label{1}
	Let   $K/k$ be a cyclic   extension of prime degree $p$. Denote by $A_{K/k}$ the number of ambiguous ideal classes of the number field $K$ with respect to $k$, then:  $$ A_{K/k} = h(k) 2^{\mu+r^{*}-(r+c+1)}$$
\noindent with:\\
	$r$ the number of the fundamental units of  $k$.\\
	$\mu$ the number of prime ideals of  $k$ $($finite or infinite$)$ which ramify in  $K$.\\
	$r^*$ is defined by $2^{r^{*}}=[N_{K/k}(K^{*})\cap E_{k}:E_{k}^2]$ with $E_{k}$ is the group of units of  $k$.\\
	$c=1$  if  $k$ contains  a primitive  $p$-th root of unity and  $c=0$ otherwise.\\
	Furthermore,  if $p=2$ and the class number of $k$ is  odd, then the  $2$-rank of the class group of $K$ is equal to
	$$\mu+r^{*}-(r+c+1).$$
\end{them}
\noindent The remark below explain how to get $r^*$.
\begin{rema}\label{14}
Since the unit group  of  $k$ is generated by $-1$ and $\epsilon_{0}$, so
\begin{enumerate}[$\bullet$]
	\item $r^*=0$,  if $-1, \epsilon_{0}$ and $-\epsilon_{0}$ are not in $N_{\mathbb{K}/k}(\mathbb{K}^{*})$.
	\item $r^*=1$, if $(-1$ is  in $N_{\mathbb{K}/k}(\mathbb{K}^{*})$ and    $\epsilon_{0}$  is not) or  $(-1$ is not   in $N_{\mathbb{K}/k}(\mathbb{K}^{*})$ and  $\epsilon_{0}$ or $-\epsilon_{0}$ is).
	\item $r^*=2$, if  $-1$ and $\epsilon_{0}$ are  in  $N_{\mathbb{K}/k}(\mathbb{K}^{*})$.
\end{enumerate}
\end{rema}
\section{The rank of $H$}\label{15}
 Let  $\ell$ be  a prime integer congruent  to $5\pmod8$ and  $n$ a square-free positive integer relatively prime to $\ell$.  Let  $\mathbb{K}=k(\sqrt{n\epsilon_{0}\sqrt{\ell}})$ and $k=\mathbb{Q}(\sqrt{\ell})$, where  $\epsilon_{0}$ is the fundamental unit of $k$.

   On one hand, as  $\ell\equiv1\pmod4$, so it is well known  (e.g., \cite{Zink 66-67}) that
$\epsilon_0=\frac{u+v\sqrt{\ell}}{2}$ and $N(\epsilon_0)=\frac{u+v\sqrt{\ell}}{2}\frac{u-v\sqrt{\ell}}{2}=\frac{u^2-v^{2}\ell}{4}=-1$. Since $u+v\sqrt{\ell}>0$, then $u-v\sqrt{\ell} < 0$. On the other hand,  the polynomial  $f(x)=x^4-mv\ell x^2+m^2\ell$ is irreducible. Indeed,  as $\ell$ divides both of $mv\ell$ and $m^2\ell$, and $\ell^2$ does not divide $m^2\ell$ since $m$ and $\ell$ are relatively prime, so the   Eisentein characterization implies our claim. Thus $f$ is  the characteristic polynomial of $\mathbb{K}$. By the Kaplansky's theorem \cite[page $68$]{Baker 13} and since  $m^2\ell[(mv\ell)^2-4m^2\ell]=m^4\ell^2(v^2\ell-4)=m^4\ell^{2}u^2 \in \mathbb{Q}^2$, the polynomial $f$ defines  a cyclic extension over $\mathbb{Q}$ of degree  $4$, i.e., $\mathbb{K}$ is a  real cyclic quartic number field and $k$ is its unique   quadratic subfield.

To determine the exact value of $r^{*}$, we will use   the norm residue symbol applied to primes of $k$ ramifying in $\mathbb{K}$.
	Note that the infinite prime ideals of $k$ don't ramify in the extension $\mathbb{K}$. Indeed, the discriminant of
   $g(X)=f(x^2)=X^2-mv\ell X+m^2\ell$ is
	$$\Delta =(mv\ell)^2-4m^2\ell=m^2\ell(v^2\ell-4)=m^2\ell u^2\geq 0,$$ hence the roots of $g$ are
	\begin{center}
$X_1=x_1^2=\frac{mv\ell-mu\sqrt{\ell}}{2}=-m\sqrt{\ell}(\frac{u-v\sqrt{\ell}}{2})\geq 0$ and
$X_2=x_2^2=\frac{mv\ell+mu\sqrt{\ell}}{2}= m\sqrt{\ell}(\frac{u+v\sqrt{\ell}}{2})=m\epsilon_0\sqrt{\ell}\geq 0$.\end{center}
From which we deduce the result claimed.

Note that,  as $n$ is relatively prime to $\ell$, so  the prime integers dividing   $n$  don't ramify in  $k$. Likewise   $2$ stay inert in  $k$ since $\ell\equiv 5\,\pmod8$.

To compute $r_2(H)$, the rank of   the $2$-class group $H$ of  $\mathbb{K}$, we will
 distinguish many cases. For this, let $p_{1}$, $p_2$, $\cdots$, $p_{t}$, $q_{1}$, $\cdots$, $q_{s}$ be  positive prime integers. Put $\delta=1$ or $2$.
\subsection{\textbf{First case: $n=\delta\prod_{i=1}^tp_{i}$ and, for all $i$,  $p_i\equiv1\pmod4$}}\label{case1}
\begin{them}\label{3}
Let $\mathbb{K}=\mathbb{Q}(\sqrt{n\epsilon_{0}\sqrt{\ell}})$ be a real cyclic quartic number field, where $\ell\equiv5\pmod8$ is a positive prime integer, $n$ a square-free positive  integer relatively prime to $\ell$ and  $\epsilon_{0}$ the fundamental unit of the quadratic subfield $k=\mathbb{Q}(\sqrt{\ell})$.
 Let  $n=\delta\prod_{i=1}^{i=t}p_{i}$ with $p_{i}\equiv 1\pmod4  $ for all $i\in\{1, \dots, t \}$ and $t$ is a positive integer.
	\begin{enumerate}[\rm1.]
		\item If, for all $i$, $(\frac{p_{i}}{\ell})=-1$, then  $r_2(H) =t$.
		\item If, for all $i$, $(\frac{p_{i}}{\ell})= 1$, then  $r_2(H) =2t$.
	\end{enumerate}
 Moreover, if $n=\delta\prod_{i=1}^{i=t_{1}}p_{i}\prod_{j=1}^{i=t_{2}}q_{j}$ with   $(\frac{p_{i}}{\ell})=-(\frac{q_{j}}{\ell})=-1$ and $p_{i}\equiv q_{j}\equiv1\,\pmod4$ for all $i\in\{1, \dots, t_{1}\}$ and  $j\in\{1, \dots, t_{2}\}$, then $r_2(H)= t_{1}+2t_{2}$.
\end{them}
\begin{proof}
  To prove the theorem assertions, we have to compute the integer $r^*$ ($2^{r^{*}}=[N_{K/k}(K^{*})\cap E_{k}:E_{k}^2]$) by applying Remark \ref{14}, and then we call Theorem \ref{1}. To deduce primes of $k$ ramifying in $\mathbb{K}$, we use Theorem \ref{wi87}.
\begin{enumerate}[\rm1.]
  \item If $(\frac{p_{i}}{\ell})=-1$  for all $i=1, \dots, t$,  then the prime ideals of $k$ which ramify in $\mathbb{K}$ are $(\sqrt{\ell})$,   $\mathfrak{2}$ and the prime ideals $\mathfrak{p}_{i}$,  $i=1,  \dots,  t$,  where  $\mathfrak{2}$ (resp. $\mathfrak{p}_{i}$) is the prime ideal of $k$ above $2$ (resp. $p_i$), this implies that the number of primes of $k$ ramifying in $\mathbb{K}$ is $\mu= t + 2$. Hence
$$\begin{array}{ll}
\left(\frac{-1, \,  d}{\mathfrak{p}_{i}}\right)=\left[\frac{-1}{\mathfrak{p}_{i}}\right]=\left(\frac{-1}{p_i}\right)=1  \text{ for all } i=1, \dots,  t. \\
\left(\frac{-1, \,  d}{\left(\sqrt{\ell}\right)}\right)=\left[\frac{-1}{\left(\sqrt{\ell}\right)}\right]=\left(\frac{-1}{\ell}\right)=1. \\
\left(\frac{\epsilon, \,  d}{\mathfrak{p}_{i}}\right)=\left[\frac{\epsilon}{\mathfrak{p}_{i}}\right]=\left(\frac{-1}{p_{i}}\right)=1
 \text{ for all }  i=1, \dots,  t. \\
\left(\frac{\epsilon, \,  d}{(\sqrt{\ell})}\right)=\left[\frac{\frac{u}{2}}{(\sqrt{\ell})}\right]=\left[\frac{2u}{(\sqrt{\ell})}\right]=\left(\frac{2u}{\ell}\right)=
\left(\frac{2}{\ell}\right)\left(\frac{u}{\ell}\right)=
-\left(\frac{\ell}{u}\right)=-1, \text{ indeed as }\\\epsilon=\frac{u+v\sqrt{\ell}}{2},  \text{ so }  -4=u^{2}-v^{2}\ell.
\end{array}$$
We summarize these results in the following table.
\begin{center}
	\begin{tabular}{|c|c|c|c|}
		\hline
		Unit $\backslash$ Character & $\sqrt{\ell}$ & $\mathfrak{p}_{i}$ & $\mathfrak{2}$\\
		\hline
		-1 &  +  & + & + \\
		\hline
		$\epsilon$ & - &+ & - \\
		\hline
		$-\epsilon$ & - & + & -\\
		\hline
	\end{tabular}
\end{center}
 So $r^{*}=1$, from which we infer that:  $$r_2(H) = \mu + r^{*} -3 = t + 2 + 1-3=t.$$
  \item If $(\frac{p_{i}}{\ell})=1$  for all $i=1, \dots, t$,  then the prime ideals of $k$ which ramify in $\mathbb{K}$ are $(\sqrt{\ell})$,   $\mathfrak{2}$ and the prime ideals $\wp_{i}$ and $\bar{\wp}_{i}$   with $p_i\mathcal{O}_k=\wp_{i}\bar{\wp}_{i}$,  $i=1, \dots, t$, in this case $\mu= 2t + 2$.  Hence
$$\left(\frac{-1, d}{\wp_{i}}\right)=\left(\frac{-1, d}{\bar{\wp}_{i}}\right)=\left[\frac{-1}{\wp_{i}}\right]=\left(\frac{-1}{p_{i}}\right)=1, \text{ for all } i=1, \dots, t.$$
$$\left(\frac{\epsilon, d}{\wp_{i}}\right)=\left[\frac{\epsilon}{\wp_{i}}\right],\text{ for all } i=1, \dots, t.$$
 To compute the last unity, put $p_i^{h_{0}}=\wp_{i}\bar{\wp}_{i}$ and  $\wp_{i}= a_{i} +b_{i} \sqrt{\ell}$ and $\bar{\wp}_{i}= a_{i} -b_{i} \sqrt{\ell},$ for all   $i$ (note that $\mathcal{O}_k$ is a principal ring). According to \cite{Parry 77}  we have
  $\left[\frac{\epsilon \sqrt{\ell}}{\wp_{i}}\right]= \left(\frac{p_i}{\ell}\right)_{4}.$  Thus
   $$\left[\frac{\epsilon }{\wp_{i}}\right]= \left(\frac{p_i}{\ell}\right)_{4}\left[\frac{ \sqrt{\ell}}{\wp_{i}}\right].$$
On the other hand, $$\left[\frac{ \sqrt{\ell}}{\wp_{i}}\right]=\left[\frac{ b_{i}^{2}\sqrt{\ell}}{\wp_{i}}\right] =\left[\frac{b_{i}(-a_{i}+a_{i} + b_{i} \sqrt{\ell})}{\wp_{i}}\right]=\left[\frac{ -a_{i}b_{i}}{\wp_{i}}\right]=\left(\frac{a_{i}}{p_i}\right)\left(\frac{b_{i}}{p_i}\right).$$
 As $p_i^{h_{0}}=a_{i}^{2} - b_{i}^{2}\ell$, so  $b_{i}^{2}\ell\equiv a_{i}^{2}\,  \pmod {p_i}$. Note that $\ell$ and  $p_i$ are relatively prime, then   $b_{i}^{2}\ell^{2}\equiv \ell a_{i}^{2}\,  \pmod {p_i}$,  this implies that  $\left(\frac{b_{i}}{p_i}\right)=\left(\frac{b_{i}\ell}{p_i}\right)=\left(\frac{b_{i}^{2}\ell^{2}}{p_i}\right)_{_{4}}=\left(\frac{\ell a_{i}^{2}}{p_i}\right)_{_{4}}=
 \left(\frac{\ell}{p_i}\right)_{4}\left(\frac{a_{i}}{p_i}\right)$.
 Finally,  $$\left[\frac{\epsilon }{\wp_{i}}\right]= \left(\frac{p_i}{\ell}\right)_{_{4}}\left(\frac{ a_{i}}{p_i}\right)\left(\frac{ b_{i}}{p_i}\right)=\left(\frac{p_i}{\ell}\right)_{_{4}}\left(\frac{ a_{i}}{p_i}\right)\left(\frac{\ell}{p_i}\right)_{_{4}}\left(\frac{ a_{i}}{p_i}\right)=\left(\frac{p_i}{\ell}\right)_{_{4}}\left(\frac{\ell}{p_i}\right)_{_{4}}.$$
Proceeding similarly,  we get   $\left[\frac{\epsilon }{\bar{\wp_{i}}}\right]=  \left(\frac{p_i}{\ell}\right)_{_{4}}\left(\frac{\ell}{p_i}\right)_{_{4}}$. The two values $\left(\frac{-1, \,  d}{(\sqrt{\ell})}\right)$ and  $\left(\frac{\epsilon,\,  d}{(\sqrt{\ell})}\right)$ are computed as above. These results are summarized in the following table:
\begin{center}
	\begin{tabular}{|c|c|c|c|c|}
		\hline
		Unit$\backslash$ Character & $\sqrt{\ell}$ & $ \wp_{i}$ & $\bar{\wp}_{i}$ & $\mathfrak{2}$\\
		\hline
		-1 &  +  & + & +  & + \\
		\hline
		$\epsilon$ & - & $\left(\frac{p_i}{\ell}\right)_{_{4}}\left(\frac{\ell}{p_i}\right)_{_{4}} $ & $\left(\frac{p_i}{\ell}\right)_{_{4}}\left(\frac{\ell}{p_i}\right)_{_{4}} $  & - \\
		\hline
		$-\epsilon$ & - & $\left(\frac{p_i}{\ell}\right)_{_{4}}\left(\frac{\ell}{p_i}\right)_{_{4}} $ & $\left(\frac{p_i}{\ell}\right)_{_{4}}\left(\frac{\ell}{p_i}\right)_{_{4}} $ & -\\
		\hline
	\end{tabular}
\end{center}
  So $r^{*}=1$, from which we infer that:  $$r_2\left(H\right) = \mu + r^{*} -3 = 2t + 2 + 1-3=2t.$$
\end{enumerate}
  If $n=\delta\prod_{i=1}^{i=t_{1}}p_{i}\prod_{j=1}^{i=t_{2}}q_{j}$ with $\left(\frac{p_{i}}{\ell}\right)=-\left(\frac{q_{j}}{\ell}\right)=-1$ for all $ i\in\{1, \dots, t_{1}\}$ and $j\in\{1, \dots, t_{2}\}$,  then according to the two cases above, $r^{*}=1$ and   $r_2\left(H\right)= t_{1}+2t_{2}\cdot$
  \end{proof}
\subsection{Second case: $n = 1$ or $2$}
\begin{them}\label{12}
Let $\mathbb{K}=\mathbb{Q}(\sqrt{n\epsilon_{0}\sqrt{\ell}})$ be a real cyclic quartic number field, where $\ell\equiv5\pmod8$ is a positive prime integer, $n$ a square-free positive  integer relatively prime to $\ell$ and  $\epsilon_{0}$ the fundamental unit of the quadratic subfield $k=\mathbb{Q}(\sqrt{\ell})$.     If $n=1$ or $2$, then $r_2\left(H\right)= 0$.
\end{them}
\begin{proof} As in the first case,  we  compute  $r^*$  by applying Remark \ref{14},  and then we call Theorem \ref{1}.
 For the two cases $n=1$ or $2$, and by Theorem \ref{wi87},  the prime ideals of  $k$ that ramify in  $\mathbb{K}$ are $(\sqrt{\ell})$ and $\mathfrak{2}$,  where $\mathfrak{2}$ is the prime ideal of $k$ above $2$, i.e. $\mu=2$. Proceeding as in the first case, we get the following table:
\begin{center}
	\begin{tabular}{|c|c|c|}
		\hline
		Unit$\backslash$ Character & $\sqrt{\ell}$ & $\mathfrak{2}$\\
		\hline
		-1 &  + & + \\
		\hline
		$\epsilon$ & -   & - \\
		\hline
		$-\epsilon$ & -  & -\\
		\hline
	\end{tabular}
\end{center}
 Hence $r^{*}=1$,  which implies that   $r_2\left(H\right) = \mu + r^{*} -3 =2+1-3=0.$
 \end{proof}
\subsection{\textbf{Third case: $n = \prod_{i=1}^{i=t}p_{i}$ with $t$ is odd and, for all $i$, $p_i\equiv3\pmod4$}}
\begin{them}\label{6}
	Let $\mathbb{K}=\mathbb{Q}(\sqrt{n\epsilon_{0}\sqrt{\ell}})$ be a real cyclic quartic number field, where $\ell\equiv5\pmod8$ is a positive prime integer, $n$ a square-free positive  integer relatively prime to $\ell$ and  $\epsilon_{0}$ the fundamental unit of the quadratic subfield $k=\mathbb{Q}(\sqrt{\ell})$. Assume  $n = \prod_{i=1}^{i=t}p_{i}$,  $p_{i}\equiv3\pmod4$ for all  $i=1, \dots, t$       and $t$ is a positive odd integer.
\begin{enumerate}[\rm1.]
	\item If, for all $i$,  $\left(\frac{p_{i}}{\ell}\right)=-1$,  then  $r_2\left(H\right) =t-1$.
	\item If, for all $i$,  $\left(\frac{p_{i}}{\ell}\right)= 1$,  then  $r_2\left(H\right) =2t-2$.
\end{enumerate}	
Moreover,  if $n=\prod_{i=1}^{t_{1}}p_{i}\prod_{j=1}^{t_{2}}q_{j}$, where  $ p_{i}\equiv q_{j}\equiv3\pmod4$ and  $\left(\frac{p_{i}}{\ell}\right)=-\left(\frac{q_{j}}{\ell}\right)=-1,$ for all  $i\in\{1, \dots, t_{1}\}$,  and for all $j\in\{1, \dots, t_{2}\}$ with $t_{1}+t_{2}$ is odd,  then  $r_2\left(H\right)=  t_{1} +2t_{2}-2$.
\end{them}
\begin{proof} As in the two cases above,  we  shall compute  $r^*$  by applying Remark \ref{14},  and then we call Theorem \ref{1}. As above Theorem \ref{wi87} gives us primes of $k$ that ramify in $\mathbb{K}$.
\begin{enumerate}[\rm1.]
\item If $\left(\frac{p_{i}}{\ell}\right)=-1$ for all  $i\in\{1, \dots, t\}$, then  the prime ideals of  $k$ which ramify in  $\mathbb{K}$ are $\mathfrak{p}_{i}$ and $(\sqrt{\ell})$, where $\mathfrak{p}_{i}$ is the prime ideal of $k$ above $p_i$. Hence, for all $i\in\{1, \dots, t\}$, we have:
$$\begin{array}{ll}
\left(\frac{-1,\, d}{\mathfrak{p}_{i}}\right)=\left[\frac{-1}{\mathfrak{p}_{i}}\right]= \left(\frac{1}{p_{i}}\right)=1 \text{ and }
\left(\frac{\epsilon,\, d}{\mathfrak{p}_{i}}\right)=\left[\frac{\epsilon}{\mathfrak{p}_{i}}\right]=\left(\frac{-1}{p_{i}}\right)= -1.
\end{array}$$
The two values $\left(\frac{-1, \,  d}{(\sqrt{\ell})}\right)$ and  $\left(\frac{\epsilon,\,  d}{(\sqrt{\ell})}\right)$ are computed as in the first case. We summarize these results in the following table:
\begin{center}
	\begin{tabular}{|c|c|c| }
		\hline
		Unit $\backslash$ Character & $(\sqrt{\ell})$&$\mathfrak{p}_{i}$ \\
		\hline
		-1 &+&  +  \\
		\hline
		$\epsilon$ & - &-   \\
		\hline
		$-\epsilon$ & - &-  \\
		\hline
	\end{tabular}
\end{center}
 Hence $r^{*}=1$,  which implies that $r_2\left(H\right) = \mu + r^{*} -3 =t+1+1-3=t-1.$
\item If $\left(\frac{p_{i}}{\ell}\right)=1$ for all  $i\in\{1, \dots, t\}$, then  the prime ideals of  $k$ which ramify in  $\mathbb{K}$ are
$(\sqrt{\ell})$,  $\wp_{i}$ and $\bar{\wp}_{i}$,  where $p_i\mathcal{O}_k=\wp_{i}\bar{\wp}_{i}$, $i=1, \dots, t$. Hence, for all $i\in\{1, \dots, t\}$, we have:
$$\left(\frac{-1,\, d}{\wp_{i}}\right)=\left(\frac{-1, d}{\bar{\wp}_{i}}\right)=\left[\frac{-1}{\wp_{i}}\right]=\left(\frac{-1}{p_{i}}\right)=-1.$$
 Proceeding as in the proof of Theorem \ref{3}, we get
 $$\left(\frac{\epsilon,\, d}{\wp_{i}}\right)=\left[\frac{\epsilon}{\wp_{i}}\right]=\left(\frac{p_i}{\ell}\right)_{4}\left(\frac{\ell}{p_i}\right)_{4}\cdot$$
  According to \cite{Parry 77}, we have $$\left[\frac{\epsilon \sqrt{\ell}}{\bar{\wp_{i}}}\right]=-\left(\frac{p}{\ell}\right)_{_{4}} \text{ then } \left[\frac{\epsilon }{\bar{\wp_{i}}}\right]=-\left(\frac{p}{\ell}\right)_{_{4}}\left[\frac{ \sqrt{\ell}}{\bar{\wp_{i}}}\right].$$
 On the other hand, $$\left[\frac{\sqrt{\ell}}{\bar{\wp_{i}}}\right]=
 -\left[\frac{ -b_{i}^{2}\sqrt{\ell}}{\bar{\wp_{i}}}\right] =-\left[\frac{b_{i}(-a_{i}+a_{i} - b_{i} \sqrt{\ell})}{\bar{\wp_{i}}}\right]
 =-\left[\frac{ -a_{i}b_{i}}{\bar{\wp_{i}}}\right]=\left(\frac{a_{i}}{p_i}\right)\left(\frac{b_{i}}{p_i}\right),$$
 which implies (as in the proof of Theorem \ref{3}) that
  $$\left[\frac{\epsilon }{\bar{\wp_{i}}}\right]= - \left(\frac{p_i}{\ell}\right)_{_{4}}\left(\frac{\ell}{p_i}\right)_{_{4}}\left(\frac{a}{p_i}\right)^{2}=-\left(\frac{p_i}{\ell}\right)_{_{4}}
  \left(\frac{\ell}{p_i}\right)_{_{4}}.$$

 The  two values $\left(\frac{-1, \,  d}{(\sqrt{\ell})}\right)$ and  $\left(\frac{\epsilon,\,  d}{(\sqrt{\ell})}\right)$ are computed as above. We summarize these results in the following table:
\begin{center}
	\begin{tabular}{|c|c|c|c|}
		\hline
		Unit/Character & $\sqrt{\ell}$ & $ \wp_{i}$ & $\bar{\wp}_{i}$ \\
		\hline
		-1 &  +  & - & - \\
		\hline
		$\epsilon$ & - & $\left(\frac{p_{i}}{\ell}\right)_{_{4}}\left(\frac{\ell}{p_{i}}\right)_{_{4}} $ & $-\left(\frac{p_{i}}{\ell}\right)_{_{4}}\left(\frac{\ell}{p_{i}}\right)_{_{4}} $   \\
		\hline
		$-\epsilon$ & - & $-\left(\frac{p_{i}}{\ell}\right)_{_{4}}\left(\frac{\ell}{p_{i}}\right)_{_{4}} $ & $\left(\frac{p_{i}}{\ell}\right)_{_{4}}\left(\frac{\ell}{p_{i}}\right)_{_{4}} $ \\
		\hline
	\end{tabular}
\end{center}
 Hence $r^{*}=0$,  which implies that:   $$r_2(H) = \mu + r^{*} -3 =2t+1+0-3=2t-2.$$
\end{enumerate}
Finally,  if $n=\prod_{i=1}^{t_{1}}p_{i}\prod_{j=1}^{t_{2}}q_{j},$ with $p_{i}\equiv q_{j}\equiv3\pmod4$ and   $\left(\frac{p_{i}}{\ell}\right)=-\left(\frac{q_{j}}{\ell}\right)=-1$ for all  $i\in\{1, \dots, t_{1}\}$ and  $j\in\{1, \dots, t_{2}\}$ with $t_{1}+t_{2}$ is odd, then according to the two cases above,  there are $t_{1} +2t_{2}+1$  prime ideals of $k$ which ramify in  $\mathbb{K}$  and $r^{*}=0$. Thus
$r_2(H)= t_{1} +2t_{2}+1+0-3 =t_{1} +2t_{2}-2.$
\end{proof}
\subsection{\textbf{Fourth case: $n =\delta\prod_{i=1}^{i=t}p_{i}$ with  $t$ is even or $n =2\prod_{i=1}^{i=t}p_{i}$ with  $t$ is odd and, for all $i$, $p_{i}\equiv 3\pmod4$}}
\begin{them}\label{7}
Let $\mathbb{K}=\mathbb{Q}(\sqrt{n\epsilon_{0}\sqrt{\ell}})$ be a real cyclic quartic number field, where $\ell\equiv5\pmod8$ is a positive prime integer, $n$ a square-free positive  integer relatively prime to $\ell$ and  $\epsilon_{0}$ the fundamental unit of the quadratic subfield $k=\mathbb{Q}(\sqrt{\ell})$. Assume  $n = \delta\prod_{i=1}^{i=t}p_{i} $ and $t$ is even or $n =2\prod_{i=1}^{i=t}p_{i}$ and  $t$ is odd with $ p_{i}\equiv 3\, \pmod4$,  $i=1, \dots, t$.
	\begin{enumerate}[\upshape1.]
		\item If, for all $i$,  $\left(\frac{p_{i}}{\ell}\right)=-1$,  then  $r_2(H) =t$.
		\item If, for all $i$,  $\left(\frac{p_{i}}{\ell}\right)= 1$, then  $r_2(H) =2t-1$.
	\end{enumerate}
Moreover,  if $n=\delta\prod_{i=1}^{t_{1}}p_{i}\prod_{j=1}^{t_{2}}q_{j}$,  with $t_{1}+t_{2}$ is even, or $n=2\prod_{i=1}^{t_{1}}p_{i}\prod_{j=1}^{t_{2}}q_{j}$ with $t_{1}+t_{2}$ is odd, where  $ p_{i}\equiv q_{j}\equiv3\pmod4$ and  $\left(\frac{p_{i}}{\ell}\right)=-\left(\frac{q_{j}}{\ell}\right)=-1,$ for all  $i\in\{1, \dots, t_{1}\}$ and  $j\in\{1, \dots, t_{2}\}$,  then  $r_2\left(H\right)=  t_{1} +2t_{2}-1$.	
\end{them}
\begin{proof} We proceed as above, we first compute  $r^*$  by applying Remark \ref{14},  and then we apply Theorem \ref{1}. It is a routine, as above, to use  Theorem \ref{wi87} to get  primes of $k$ that ramify in $\mathbb{K}$.
\begin{enumerate}[\rm1.]
\item If $\left(\frac{p_{i}}{\ell}\right)=-1$ for all  $i\in\{1, \dots, t\}$, then  the prime ideals of  $k$ which ramify in  $\mathbb{K}$ are $\mathfrak{2}$,  $\mathfrak{p}_{i}$ and $(\sqrt{\ell})$, where $\mathfrak{p}_{i}$ (resp. $\mathfrak{2}$) is the prime ideal of $k$ above $p_i$ (resp. $2$). Proceeding as in the cases above, we get the following table.
\begin{center}
	\begin{tabular}{|c|c|c|c| }
		\hline
		Unit/Character & $\sqrt{\ell}$& $\mathfrak{2}$ & $\mathfrak{p_{i}}$ \\
		\hline
		-1 &+&  + & +  \\
		\hline
		$\epsilon$ & - & $(-1)^{t+1}$ & -  \\
		\hline
		$-\epsilon$ & - & $(-1)^{t+1}$ & - \\
		\hline
	\end{tabular}
\end{center}
 Hence $r^{*}=1$, from which we deduce that $$r_2(H) = \mu + r^{*} -3 =t+ 2 +1-3=t.$$
\item If $\left(\frac{p_{i}}{\ell}\right)=1$ for all  $i\in\{1, \dots, t\}$, then  the prime ideals of $k$ which ramify in  $\mathbb{K}$ are $(\sqrt{\ell})$,  $\mathfrak{2}$,  $\wp_{i}$,  and $\bar{\wp}_{i}$, where $p\mathcal{O}_k=\wp_{i}\bar{\wp}_{i}$,  $i=1, \dots, t$. Proceeding as in the cases above, we get the following table.
\begin{center}
	\begin{tabular}{|c|c|c|c|c|}
		\hline
		Unit/Character & $\sqrt{\ell}$ & $ \wp_{i}$ & $\bar{\wp}_{i}$ & $\mathfrak{2}$\\
		\hline
		-1 &  +  & - & -  & + \\
		\hline
		$\epsilon$ & - & $\left(\frac{p_i}{\ell}\right)_{_{4}}\left(\frac{\ell}{p_i}\right)_{_{4}} $ & $-\left(\frac{p_i}{\ell}\right)_{_{4}}\left(\frac{\ell}{p_i}\right)_{_{4}} $  &  $(-1)^{t+1}$ \\
		\hline
		$-\epsilon$ & - & $-\left(\frac{p_i}{\ell}\right)_{_{4}}\left(\frac{\ell}{p_i}\right)_{_{4}} $ & $\left(\frac{p_i}{\ell}\right)_{_{4}}\left(\frac{\ell}{p_i}\right)_{_{4}} $ &  $(-1)^{t+1}$\\
		\hline
	\end{tabular}
\end{center}
 Hence $r^{*}=0$, from which we deduce that $$r_2(H) = \mu + r^{*} -3 = 2t + 2 + 0 -3=2t-1.$$
\end{enumerate}
Assume  $n=\delta\prod_{i=1}^{t_{1}}p_{i}\prod_{j=1}^{t_{2}}q_{j}$,  with $t_{1}+t_{2}$ is even, or $n=2\prod_{i=1}^{t_{1}}p_{i}\prod_{j=1}^{t_{2}}q_{j}$, with $t_{1}+t_{2}$ is odd, where  $ p_{i}\equiv q_{j}\equiv3\pmod4$ and  $\left(\frac{p_{i}}{\ell}\right)=-\left(\frac{q_{j}}{\ell}\right)=-1,$ for all  $i\in\{1, \dots, t_{1}\}$  and  $j\in\{1, \dots, t_{2}\}$, then according to the previous discussion we obtain  $r^{*}=0$, and thus $r_2\left(H\right)=  t_{1} +2t_{2}-1.$	
\end{proof}
\subsection{\textbf{Fifth case: $n =\prod_{i=1}^{i=t}p_{i}\prod _{j=1}^{j=s}q_{j}$,  $p_{i}\equiv-q_{j}\equiv1\pmod4$  $\forall (i, j)$ and $s$ is odd}}
\begin{them}
Let $\mathbb{K}=\mathbb{Q}(\sqrt{n\epsilon_{0}\sqrt{\ell}})$ be a real cyclic quartic number field, where $\ell\equiv5\pmod8$ is a positive prime integer, $n$ a square-free positive  integer relatively prime to $\ell$ and  $\epsilon_{0}$ the fundamental unit of the quadratic subfield $k=\mathbb{Q}(\sqrt{\ell})$. Assume $n =\prod_{i=1}^{i=t}p_{i}\prod _{j=1}^{j=s}q_{j}$ with  $p_{i}\equiv -q_{j}\equiv 1\pmod4$ for all  $(i, j)\in\{1, \dots, t\}\times\{1, \dots, s\}$ and $s$ is odd. Denote by $h$ the number of prime ideals of $k$ dividing all the  $p_{i}'s$,  $i\in\{1, \dots, t\}$.
	\begin{enumerate}[\rm1.]
	\item If, for all $j$,  $\left(\frac{q_{j}}{\ell}\right)=-1$, then  $r_2(H) =h+s-1$.
	\item If, for all $j$, $\left(\frac{q_{j}}{\ell}\right)=1$, then  $r_2(H) =h+2s-2$.
    \end{enumerate}
Moreover,   if  $\displaystyle\prod _{j=1}^{j=s}q_{j}=\prod_{i=1}^{i=s_1}q_{i}\prod _{j=s_1+1}^{j=s}q_{j}$ with   $s$ odd, $q_{i}\equiv q_{j}\equiv3\, \pmod4$ and $\left(\frac{q_{i}}{\ell}\right)=-\left(\frac{q_{j}}{\ell}\right)=-1$, for all $i\in\{1, \dots, s_1\}$ and  $j\in\{s_1+1, \dots, s\}$,   then  $r_2(H)=h+s_{1}+2(s-s_{1})-2$.
\end{them}
\begin{proof}Proceeding as above, we have  four cases to discus.
\begin{enumerate}[\rm1.]
	\item  If $\left(\frac{p_{i}}{\ell}\right)=\left(\frac{q_{j}}{\ell}\right)=-1$, for all $i=1, \dots, t$ and for all $j=1, \dots, s$, then   the prime ideals of  $k$ which ramify in $\mathbb{K}$ are $(\sqrt{\ell})$, $\mathfrak{ p}_{i}$ and $\mathfrak{q}_{j}$,  the prime ideals of $k$ above ${ p}_{i}$ and ${q}_{j}$ respectively, i.e. $\mu= t+s+1$. Proceeding as  above, we get the following table.
\begin{center}
	\begin{tabular}{|c|c|c|c|}
		\hline
		Unit/Character & $\sqrt{\ell}$ & $\mathfrak{p}_{i}$ & $\mathfrak{q}_{j}$\\
		\hline
		-1 &  + & + & + \\
		\hline
		$\epsilon$ & - & +  & - \\
		\hline
		$-\epsilon$ & - & + & -\\
		\hline
	\end{tabular}
\end{center}
 Thus $r^{*}=1$, from which we deduce  that $$r_2(H) = \mu + r^{*} -3 = t+s+1 + 1-3=t+s-1.$$
\item If $\left(\frac{p_{i}}{\ell}\right)=-\left(\frac{q_{j}}{\ell}\right)=1$, for all $i=1, \dots, t$ and for all $j=1, \dots, s$, then  the prime ideals of  $k$ which ramify in  $\mathbb{K}$ are $(\sqrt{\ell})$,  $\wp_{i}$,  $\bar{\wp}_{i}$ and $\mathfrak{q}_{j}$,  where  $p\mathcal{O}_k=\wp_{i}\bar{\wp}_{i}$ and $\mathfrak{q}_{j}$ is the prime ideal of $k$ above  ${q}_{j}$, i.e. $\mu= 2t+s+1$.  Proceeding as  above, we get the following results:
\begin{center}
	\begin{tabular}{|c|c|c|c|c|}
		\hline
		Unit/Character & $\sqrt{\ell}$ & $\wp_{i}$ & $\bar{\wp}_{i}$ & $\mathfrak{q}_{j}$\\
		\hline
		-1 &  + & + & + &+\\
		\hline
		$\epsilon$ & - & $\left(\frac{p_i}{\ell}\right)_{_{4}}\left(\frac{\ell}{p_{i}}\right)_{_{4}} $ & $\left(\frac{p_i}{\ell}\right)_{_{4}}\left(\frac{\ell}{p_i}\right)_{_{4}} $ & - \\
		\hline
		$-\epsilon$ & - & $\left(\frac{p_i}{\ell}\right)_{_{4}}\left(\frac{\ell}{p_i}\right)_{_{4}} $ & $\left(\frac{p_i}{\ell}\right)_{_{4}}\left(\frac{\ell}{p_i}\right)_{_{4}} $ & -\\
		\hline
	\end{tabular}
\end{center}
 Thus $r^{*}=1$, this result implies that $$r_2(H) = \mu + r^{*} -3 = 2t+s+1 + 1-3=2t+s-1.$$

  \item If $\left(\frac{p_{i}}{\ell}\right)=-\left(\frac{q_{j}}{\ell}\right)=-1$, for all $i=1, \dots, t$ and  for all $j=1, \dots, s$, then  the prime ideals of  $k$ which ramify in  $\mathbb{K}$ are $(\sqrt{\ell})$,  $\mathfrak{p}_{i}$,  $\rho_{j}$,  and $\bar{\rho}_{j}$, where   $q_j\mathcal{O}_k=\rho_{j}\bar{\rho}_{j}$ and $\mathfrak{p}_{j}$ is the prime ideal of $k$ above  ${p}_{j}$, i.e. $\mu= t+2s+1$.  Proceeding as  above, we get the following table:
 \begin{center}
 	\begin{tabular}{|c|c|c|c|c|}
 		\hline
 		Unit/Character & $\sqrt{\ell}$ & $\mathfrak{p}_{i}$  & $\rho_{j}$&$\bar{\rho}_{j}$\\
 		\hline
 		-1 &  +& +  &-&-\\
 		\hline
 		$\epsilon$ & - & + & $\left(\frac{q_j}{\ell}\right)_{_{4}}\left(\frac{\ell}{q_{j}}\right)_{_{4}} $&-$\left(\frac{q_j}{\ell}\right)_{_{4}}\left(\frac{\ell}{q_{j}}\right)_{_{4}} $ \\
 		\hline
 		$-\epsilon$ & - & + & $-\left(\frac{q_j}{\ell}\right)_{_{4}}\left(\frac{\ell}{q_{j}}\right)_{_{4}} $&$\left(\frac{q_j}{\ell}\right)_{_{4}}\left(\frac{\ell}{q_{j}}\right)_{_{4}} $\\
 		\hline
 	\end{tabular}
 \end{center}
 Thus $r^{*}=0$, this implies that $$r_2(H) = \mu + r^{*} -3 = t+2s+1 + 0-3=t+2s-2.$$

 \item If $\left(\frac{p_{i}}{\ell}\right)=\left(\frac{q_{j}}{\ell}\right)=1$, for all $i=1, \dots, t$ and for all $j=1, \dots, s$, then   the prime ideals of  $k$ which ramify in $\mathbb{K}$ are $ (\sqrt{\ell})$,  $\wp_{i}$,  $\bar{\wp}_{i}$,  $\rho_{j}$ and $\bar{\rho}_{j}$,  where $p_i\mathcal{O}_k=\wp_{i}\bar{\wp}_{i}$ and $q_j\mathcal{O}_k=\rho_{j}\bar{\rho}_{j}$, i.e. $\mu= 2t+2s+1$.  Proceeding as  above, we get the following table:
 \begin{center}
 	\begin{tabular}{|c|c|c|c|c|c|}
 		\hline
 		Unit/Character & $\sqrt{\ell}$ & $\wp_{i}$ & $\bar{\wp}_{i}$ & $\rho_{j}$&$\bar{\rho}_{j}$\\
 		\hline
 		-1 &  +& + & + &-&-\\
 		\hline
 		$\epsilon$ & - &$\left(\frac{p_i}{\ell}\right)_{_{4}}\left(\frac{\ell}{p_i}\right)_{_{4}} $ & $-\left(\frac{p_i}{\ell}\right)_{_{4}}\left(\frac{\ell}{p_i}\right)_{_{4}} $& $\left(\frac{q_j}{\ell}\right)_{_{4}}\left(\frac{\ell}{q_j}\right)_{_{4}} $ & $\left(\frac{q_j}{\ell}\right)_{_{4}}\left(\frac{\ell}{q_j}\right)_{_{4}} $  \\
 		\hline
 		$-\epsilon$ & - &$-\left(\frac{p_i}{\ell}\right)_{_{4}}\left(\frac{\ell}{p_i}\right)_{_{4}} $ & $\left(\frac{p_i}{\ell}\right)_{_{4}}\left(\frac{\ell}{p_i}\right)_{_{4}} $ &$\left(\frac{q_j}{\ell}\right)_{_{4}}\left(\frac{\ell}{q_j}\right)_{_{4}} $ & $\left(\frac{q_j}{\ell}\right)_{_{4}}\left(\frac{\ell}{q_j}\right)_{_{4}} $   \\
 		\hline
 	\end{tabular}
 \end{center}
 From which we deduce that $r^{*}=0$, and thus  $$r_2(H) = \mu + r^{*} -3 = 2t+2s+1 + 0-3=2t+2s-2.$$
\end{enumerate}
In general,   if $\displaystyle\prod _{j=1}^{j=s}q_{j}=\prod_{i=1}^{i=s_1}q_{i}\prod _{j=s_1+1}^{j=s}q_{j}$, with   $s$   is an odd integer,  $q_{i}\equiv q_{j}\equiv3\, \pmod4$ and $\left(\frac{q_{i}}{\ell}\right)=-\left(\frac{q_{j}}{\ell}\right)=-1,  i\in\{1, \dots, s_{1}\},  j\in\{1, \dots, s_{2}\}$, then taking into account the discussions above one gets $r^{*}=0$, and thus
$$r_2(H)= h+s_{1}+2s_{2} + 1 + 0 -3=h+s_{1}+2(s-s_{1})-2,$$ where  $h$ is always the number of the prime divisors of all the $p_{i}$'s,    $p_{i}\equiv1\, \pmod4,  i\in\{1, \dots, t\}$ in  $k$.
\end{proof}
\subsection{\textbf{Sixth case: $n =\delta\prod_{i=1}^{i=t}p_{i}\prod _{j=1}^{j=s}q_{j}$, $s$ is even, or $n =2\prod_{i=1}^{i=t}p_{i}\prod _{j=1}^{j=s}q_{j}$, $s$ is odd, where $ p_{i}\equiv -q_{j}\equiv 1\pmod4 $ for all $ (i,j)$}}
\begin{them}
Let $\mathbb{K}=\mathbb{Q}(\sqrt{n\epsilon_{0}\sqrt{\ell}})$ be a real cyclic quartic number field, where $\ell\equiv5\pmod8$ is a positive prime integer, $n$ a square-free positive  integer relatively prime to $\ell$ and  $\epsilon_{0}$ the fundamental unit of the quadratic subfield $k=\mathbb{Q}(\sqrt{\ell})$. Assume  $n =\displaystyle\delta\prod_{i=1}^{i=t}p_{i}\prod _{j=1}^{j=s}q_{j}$  with $s$  even or  $n =2\displaystyle\prod_{i=1}^{i=t}p_{i}\prod _{j=1}^{j=s}q_{j}$  with $s$ odd, where $ p_{i}\equiv- q_{j}\equiv 1 \pmod4$, for all $(i, j)\in\{1, \dots, t\}\times\{1, \dots, s\}$, are prime integers. Denote by $h$ the number of the prime ideals of $k$ above all the  $p_{i}'s$,  $i\in\{1, \dots, t\}$.
	\begin{enumerate}[\rm1.]
	\item If, for all $j$,  $\left(\frac{q_{j}}{\ell}\right)=-1$,  then  $r_2(H) =h+s$.
	\item If, for all $j$, $\left(\frac{q_{j}}{\ell}\right)=1$, then  $r_2(H) =h+2s-1$.
\end{enumerate}
Moreover, if $\displaystyle\prod _{j=1}^{j=s}q_{j}=\prod_{i=1}^{i=s_1}q_{i}\prod _{j=s_1+1}^{j=s}q_{j}$ with   $\left(\frac{q_{i}}{\ell}\right)=-\left(\frac{q_{j}}{\ell}\right)=-1$, for all $i\in\{1, \dots, s_1\}$ and  $j\in\{s_1+1, \dots, s\}$,  assuming $s$  even  if  $n =\displaystyle\delta\prod_{i=1}^{i=t}p_{i}\prod _{j=1}^{j=s}q_{j}$ and  odd if $n =\displaystyle 2\prod_{i=1}^{i=t}p_{i}\prod _{j=1}^{j=s}q_{j}$, then   $r_2(H)= h+s_{1}+2(s-s_{1})-1.$
\end{them}
\begin{proof}
There are also four cases to distinguish:
\begin{enumerate}[\rm1.]
\item If $\left(\frac{p_{i}}{\ell}\right)=\left(\frac{q_{j}}{\ell}\right)=-1$, for all  $i\in\{1, \dots, t\}$ and  $j\in\{1, \dots, s\}$, then   the prime ideals of  $k$ which ramify in  $\mathbb{K}$ are $\mathfrak{2}$,  $(\sqrt{\ell})$,  $\mathfrak{p}_{i}$ and $\mathfrak{q}_{j}$.  Proceeding as in the first cases  above, we get the following table:
\begin{center}
	\begin{tabular}{|c|c|c|c|c|}
		\hline
		Unit/Character & $\sqrt{\ell}$&$\mathfrak{2}$ & $\mathfrak{p}_{i}$ & $\mathfrak{q}_{j}$\\
		\hline
		-1 & +& + & + & + \\
		\hline
		$\epsilon$ & - &$(-1)^{s-1}$& +  & - \\
		\hline
		$-\epsilon$ & - &$(-1)^{s-1}$& + & -\\
		\hline
	\end{tabular}
\end{center}
Hence $r^{*}=1$, and $$r_2(H) = \mu + r^{*} -3 = t+s+2 + 1-3=t+s.$$
\item If $\left(\frac{p_{i}}{\ell}\right)=-\left(\frac{q_{j}}{\ell}\right)=1$, for all  $i\in\{1, \dots, t\}$ and   $j\in\{1, \dots, s\}$, then    the prime ideals of  $k$ which ramify in  $\mathbb{K}$ are $\mathfrak{2}$,  $(\sqrt{\ell})$,  $\wp_{i}$,  $\bar{\wp}_{i}$ and $\mathfrak{q}_{j}$, where $p_i\mathcal{O}_k=\wp_{i}\bar{\wp}_{i}$. Proceeding as in the first cases  above, we get the following table:
\begin{center}
	\begin{tabular}{|c|c|c|c|c|c|}
		\hline
		Unit/Character & $\sqrt{\ell}$ &$\mathfrak{2}$& $\wp_{i}$ & $\bar{\wp}_{i}$ & $\mathfrak{q}_{j}$\\
		\hline
		-1 &  + &+& + & + &+\\
		\hline
		$\epsilon$ & - &$(-1)^{s-1}$& $\left(\frac{p_i}{\ell}\right)_{_{4}}\left(\frac{\ell}{p_i}\right)_{_{4}} $ & $\left(\frac{p_i}{\ell}\right)_{_{4}}\left(\frac{\ell}{p_i}\right)_{_{4}} $ & - \\
		\hline
		$-\epsilon$ & - & $(-1)^{s-1}$ & $\left(\frac{p_i}{\ell}\right)_{_{4}}\left(\frac{\ell}{p_i}\right)_{_{4}} $ & $\left(\frac{p_i}{\ell}\right)_{_{4}}\left(\frac{\ell}{p_i}\right)_{_{4}} $ & -\\
		\hline
	\end{tabular}
\end{center}
Hence $r^{*}=1$, and $$r_2(H) = \mu + r^{*} -3 = 2t+s+2 + 1-3=2t+s.$$
\item If $\left(\frac{p_{i}}{\ell}\right)=-\left(\frac{q_{j}}{\ell}\right)=-1$; for all  $i\in\{1, \dots, t\}$ and   $j\in\{1, \dots, s\}$, then  the prime ideals of  $k$ which ramify in  $\mathbb{K}$ are $\mathfrak{2}$,  $(\sqrt{\ell})$,  $\mathfrak{p}_{i}$,  $\rho_{j}$,  and $\bar{\rho}_{j}$, where $q_j\mathcal{O}_k=\rho_{j}\bar{\rho}_{j}$. Proceeding as in the first cases  above, we get the following table:
\begin{center}
	\begin{tabular}{|c|c|c|c|c|c|}
		\hline
		Unit/Character & $\sqrt{\ell}$ &$\mathfrak{2}$& $\mathfrak{p}_{i}$  & $\rho_{j}$&$\bar{\rho}_{j}$\\
		\hline
		-1 &  + &+& +  &-&-\\
		\hline
		$\epsilon$ & - &$(-1)^{s-1}$& + & $\left(\frac{q_j}{\ell}\right)_{_{4}}\left(\frac{\ell}{q _j}\right)_{_{4}} $&$-\left(\frac{q_j}{\ell}\right)_{_{4}}\left(\frac{\ell}{q_{j}}\right)_{_{4}} $ \\
		\hline
		$-\epsilon$ & - &$(-1)^{s-1}$& + & $-\left(\frac{q_{j}}{\ell}\right)_{_{4}}\left(\frac{\ell}{q_{j}}\right)_{_{4}} $&$\left(\frac{q_{j}}{\ell}\right)_{_{4}}\left(\frac{\ell}{q_{j}}\right)_{_{4}} $\\
		\hline
	\end{tabular}
\end{center}
 Hence $r^{*}=0$, and $$r_2(H) = \mu + r^{*} -3 = t+2s+2 + 0-3=t+2s-1.$$
\item If $\left(\frac{p_{i}}{\ell}\right)=\left(\frac{q_{j}}{\ell}\right)=1$, for all $i\in\{1, \dots, t\}$ and   $j\in\{1, \dots, s\}$, then   the prime ideals of  $k$ which ramify in $\mathbb{K}$ are $\mathfrak{2}$,  $(\sqrt{\ell})$,  $\wp_{i}$,  $\bar{\wp}_{i}$,  $\rho_{j}$ and $\bar{\rho}_{j}$,  where $p_i\mathcal{O}_k=\wp_{i}\bar{\wp}_{i}$,  $q_j\mathrm{O}_k=\rho_{j}\bar{\rho}_{j}$. Proceeding as in the first cases  above, we get the following table:
\begin{center}
	\begin{tabular}{|c|c|c|c|c|c|c|}
		\hline
		Unit/Chara & $\sqrt{\ell}$ &$\mathfrak{2}$& $\wp_{i}$ & $\bar{\wp}_{i}$ & $\rho_{j}$&$\bar{\rho}_{j}$\\
		\hline
		-1 &  + &+& + & + &-&-\\
		\hline
		$\epsilon$ & - &$(-1)^{s-1}$& $\left(\frac{p_{i}}{\ell}\right)_{_{4}}\left(\frac{\ell}{p_{i}}\right)_{_{4}} $ & $\left(\frac{p_{i}}{\ell}\right)_{_{4}}\left(\frac{\ell}{p_{i}}\right)_{_{4}} $ & $\left(\frac{q_{j}}{\ell}\right)_{_{4}}\left(\frac{\ell}{q_{j}}\right)_{_{4}} $ & $-\left(\frac{q_{j}}{\ell}\right)_{_{4}}\left(\frac{\ell}{q_{j}}\right)_{_{4}} $\\
		\hline
		$-\epsilon$ & - &$(-1)^{s-1}$& $\left(\frac{p_{i}}{\ell}\right)_{_{4}}\left(\frac{\ell}{p_{i}}\right)_{_{4}} $ & $\left(\frac{p_{i}}{\ell}\right)_{_{4}}\left(\frac{\ell}{p_{i}}\right)_{_{4}} $ & $-\left(\frac{q_{j}}{\ell}\right)_{_{4}}\left(\frac{\ell}{q_{j}}\right)_{_{4}} $ & $\left(\frac{q_{j}}{\ell}\right)_{_{4}}\left(\frac{\ell}{q_{j}}\right)_{_{4}} $\\
		\hline
	\end{tabular}
\end{center}
 Hence $r^{*}=0$,  and $$r_2(H) = \mu + r^{*} -3 = 2t+2s+2 + 0-3=2t+2s-1.$$
\end{enumerate}
In general,  if $\displaystyle\prod _{j=1}^{j=s}q_{j}=\prod_{i=1}^{i=s_1}q_{i}\prod _{j=s_1+1}^{j=s}q_{j}$ with   $\left(\frac{q_{i}}{\ell}\right)=-\left(\frac{q_{j}}{\ell}\right)=-1$, for all $i\in\{1, \dots, s_1\}$ and  $j\in\{s_1+1, \dots, s\}$,  assuming $s$  even  if  $n =\displaystyle\delta\prod_{i=1}^{i=t}p_{i}\prod _{j=1}^{j=s}q_{j}$ and  odd if $n =\displaystyle 2\prod_{i=1}^{i=t}p_{i}\prod _{j=1}^{j=s}q_{j}$, then  taking into account the discussions above one gets $r^{*}=0$, and thus
 $r_2(H)= h+s_{1}+2(s-s_{1})+2+0-3 = h+s_{1}+2(s-s_{1})-1.$
\end{proof}

\section{Applications}\label{16}
In this section, we shall determine the integers  $n$ satisfying $r_2(H)$, the rank of the $2$-class group  $H$ of $\mathbb{K}=\mathbb{Q}(\sqrt{n\epsilon_{0}\sqrt{\ell}})$,  is equal to $0$, $1$, $2$ or $3$. For this we adopt the following notations: $p$ and $p_i$ (resp. $q$ and $q_i$), $i\in\NN^*$, are prime integers  congruent to $1$ (resp. $3$) modulo $4$. $\delta=1$ or $2$. The following theorems are simple deductions from the results of the previous sections.
 \begin{them}\label{11}
	Let $\mathbb{K}=\mathbb{Q}(\sqrt{n\epsilon_{0}\sqrt{\ell}})$ be a real cyclic quartic number field, where $\ell\equiv5\pmod8$ is a positive prime integer, $n$ a square-free positive  integer relatively prime to $\ell$ and  $\epsilon_{0}$ the fundamental unit of the quadratic subfield $k=\mathbb{Q}(\sqrt{\ell})$. The class number of $\mathbb{K}$ is odd if and only if one of the following assertions holds:
	\begin{enumerate}[\rm1.]
		\item $n=1$ or $2$.
		\item $n$ is a prime integer congruent to $3\pmod4$.	
	\end{enumerate}
\end{them}
\begin{exams} For all the examples below, we use PARI/GP calculator version 2.9.1 (64bit), Nov 22, 2016.
	\begin{enumerate}[\rm1.]
		\item For  $n=1$ and  $\ell=173\equiv5\pmod8$,   $H$ is trivial, the class number of the class group of  $\mathbb{K}=\mathbb{Q}(\sqrt{\epsilon_{0}\sqrt{\ell}})$ is $5$.\\
		For  $n=2$ and  $\ell=197\equiv5\pmod8$,  $H$ is trivial, in  reality  the class group of  $\mathbb{K}=\mathbb{Q}(\sqrt{2\epsilon_{0}\sqrt{\ell}})$ is of type $(3, 3)$.
		\item For  $n=q=67\equiv3\pmod4$ and  $\ell=53\equiv5\pmod8$, we have $(\frac{q}{\ell})=-1$  and $H$ is trivial, the class number of the class group of  $\mathbb{K}=\mathbb{Q}(\sqrt{67\epsilon_{0}\sqrt{\ell}})$ is 17.\\
		For  $n=q=79\equiv3\pmod4$ and  $\ell=13\equiv5\pmod8$, we have $(\frac{q}{\ell})=1$  and  the class group of  $\mathbb{K}=\mathbb{Q}(\sqrt{79\epsilon_{0}\sqrt{\ell}})$ is trivial.
	\end{enumerate}
\end{exams}
 \begin{them}\label{8}
	Let $\mathbb{K}=\mathbb{Q}(\sqrt{n\epsilon_{0}\sqrt{\ell}})$ be a real cyclic quartic number field, where $\ell\equiv5\pmod8$ is a prime, $n$ a square-free positive  integer relatively prime to $\ell$ and  $\epsilon_{0}$ the fundamental unit of the quadratic subfield $k=\mathbb{Q}(\sqrt{\ell})$.
 $H$ is cyclic  if and only if one of the following assertions holds:
	\begin{enumerate}[\rm1.]
		\item $n=\delta p$ with $(\frac{p}{\ell})=-1$.
		\item $n=2q$.
		\item $n=pq$ with $(\frac{p}{\ell})=-1$.	
	\end{enumerate}
\end{them}
\begin{exams}Here are some examples.
\begin{enumerate}[\rm1.]
 \item For  $n=p=13\equiv1\pmod4$ and  $\ell=37\equiv5\pmod8$, we have $(\frac{p}{\ell})=-1$  and $H$ is cyclic of order $2$.\\
 For  $n=p=17\equiv1\pmod4$ and  $\ell=29\equiv5\pmod8$, we have $(\frac{p}{\ell})=-1$  and $H$ is cyclic of order $2$.
 \item For  $n=2p=2.41\equiv2\pmod4$ and  $\ell=13\equiv5\pmod8$, we have $(\frac{p}{\ell})=-1$  and $H$ is cyclic of order $2$.\\
   For  $n=2p=2.51\equiv2\pmod4$ and  $\ell=61\equiv5\pmod8$, we have $(\frac{p}{\ell})=-1$  and $H$ is cyclic of order $2$.
\item For  $n=2q=2.19\equiv-2\pmod4$ and  $\ell=53\equiv5\pmod8$, we have  $H$ is cyclic of order $2$.
\item For  $n=pq=5.7\equiv3\pmod4$ and  $\ell=53\equiv5\pmod8$, we have $(\frac{p}{\ell})=-1$, $(\frac{p}{q})=-1$ and $H$ is cyclic of order $2$.\\
For  $n=pq=5.11\equiv3\pmod4$ and  $\ell=13\equiv5\pmod8$, we have $(\frac{p}{\ell})=-1$, $(\frac{p}{q})=1$ and $H$ is cyclic of order $2$.
   \end{enumerate}
\end{exams}
\begin{them}\label{9}
Let $\mathbb{K}=\mathbb{Q}(\sqrt{n\epsilon_{0}\sqrt{\ell}})$ be a real cyclic quartic number field, where $\ell\equiv5\pmod8$ is a prime, $n$ a square-free positive  integer relatively prime to $\ell$ and  $\epsilon_{0}$ the fundamental unit of the quadratic subfield $k=\mathbb{Q}(\sqrt{\ell})$.
 The rank $r_2(H)$ equals $2$  if and only if $n$  takes one of the following forms.
\begin{enumerate}[\rm1.]
  \item $n=\delta p$  and  $(\frac{p}{\ell})=1$.
  \item $n=\delta p_1p_2$  and  $(\frac{p_1}{\ell})=(\frac{p_2}{\ell})=-1$.
  \item $n=pq$ and  $(\frac{p}{\ell})=1$.
  \item $n=2pq$ and $(\frac{p}{\ell})=-1$.
  \item $n=p_1p_2q$ and   $(\frac{p_1}{\ell})=(\frac{p_2}{\ell})=-1$.
  \item $n=\delta q_1q_2$ and at least one of the two symbols $(\frac{q_1}{\ell})$, $(\frac{q_2}{\ell})$ equals $-1$.
 \item $n= q_1q_2q_3$ and at most one of the  symbols $(\frac{q_1}{\ell})$, $(\frac{q_2}{\ell})$, $(\frac{q_3}{\ell})$ equals $1$.
\end{enumerate}
\end{them}
\begin{exams}Here are some examples.
\begin{enumerate}[\rm1.]
 \item For  $n=p=13\equiv1\pmod4$ and  $\ell=101\equiv5\pmod8$, we have $(\frac{p}{\ell})=1$  and $H$ is of type $(2, 2)$.\\
 For  $n=2p=2.73\equiv2\pmod4$ and  $\ell=109\equiv5\pmod8$, we have $(\frac{p}{\ell})=1$  and $H$ is of type $(2, 4)$.
 \item For  $n=p_1p_2=17.37\equiv1\pmod4$ and  $\ell=29\equiv5\pmod8$, we have $(\frac{p_1}{\ell})=-1$, $(\frac{p_2}{\ell})=-1$   and $H$ is of type $(2, 2)$.
 \item For  $n=pq=13.11\equiv3\pmod4$ and  $\ell=53\equiv5\pmod8$, we have $(\frac{p}{\ell})=1$  and $H$ is of type $(2, 2)$.
\item For  $n=2pq=2.17.11\equiv-2\pmod4$ and  $\ell=29\equiv5\pmod8$, we have $(\frac{p}{\ell})=-1$ and $H$ is of type $(2, 2)$.
\item For  $n=p_1p_2q=17.37.23\equiv3\pmod4$ and  $\ell=61\equiv5\pmod8$, we have $(\frac{p_1}{\ell})=-1$, $(\frac{p_2}{\ell})=-1$ and $H$ is of type $(2, 2)$.
\item For  $n=q_1q_2=79.83\equiv1\pmod4$ and  $\ell=37\equiv5\pmod8$, we have $(\frac{q_1}{\ell})=-1$, $(\frac{q_2}{\ell})=1$ and $H$ is of type $(2, 2)$.
For  $n=2q_1q_2=2.47.59\equiv2\pmod4$ and  $\ell=13\equiv5\pmod8$, we have $(\frac{q_1}{\ell})=-1$, $(\frac{q_2}{\ell})=-1$ and $H$ is of type $(2, 2)$.
\item For  $n=q_1q_2q_3=23.71.83\equiv3\pmod4$ and  $\ell=61\equiv5\pmod8$, we have $(\frac{q_1}{\ell})=-1$, $(\frac{q_2}{\ell})=-1$, $(\frac{q_3}{\ell})=1$ and $H$ is of type $(2, 2)$.
   \end{enumerate}
\end{exams}
\begin{them}\label{10}
Let $\mathbb{K}=\mathbb{Q}(\sqrt{n\epsilon_{0}\sqrt{\ell}})$ be a real cyclic quartic number field, where $\ell\equiv5\pmod8$ is a prime, $n$ a square-free positive  integer relatively prime to $\ell$ and  $\epsilon_{0}$ the fundamental unit of the quadratic subfield $k=\mathbb{Q}(\sqrt{\ell})$.
 The rank $r_2(H)$ equals $3$  if and only if $n$  takes one of the following forms.
\begin{enumerate}[\rm1.]
  \item $n=\delta p_1p_2$ and $(\frac{p_1}{\ell})=-(\frac{p_2}{\ell})=1$.
  \item $n=\delta p_1p_2p_3$ and  $(\frac{p_i}{\ell})=-1$ for all $i\in\{1, 2, 3\}$.
  \item $n=2pq$ and $(\frac{p}{\ell})=1$.
  \item $n=\delta q_1q_2$ and  $(\frac{q_i}{\ell})=1$ for all $i\in\{1, 2\}$.
 \item $n= q_1q_2q_3$ and only one of the symbols $(\frac{q_i}{\ell})$, $i\in\{1, 2, 3\}$, is $-1$.
\item $n= 2q_1q_2q_3$ and at most one of the  symbols $(\frac{q_i}{\ell})$, $i\in\{1, 2, 3\}$,  is $1$.
 \item $n=p_1p_2q$ and   $(\frac{p_1}{\ell})=-(\frac{p_2}{\ell})=1$.
 \item $n=2p_1p_2q$ and   $(\frac{p_1}{\ell})=(\frac{p_2}{\ell})=-1$.
 \item $n= \delta pq_1q_2$ and $(\frac{p}{\ell})=-1$ and at least one of the  symbols $(\frac{q_i}{\ell})$, $i\in\{1, 2\}$,  is $-1$.
  \item $n= pq_1q_2q_3$ and $(\frac{p}{\ell})=-1$ and at most one of the  symbols $(\frac{q_i}{\ell})$, $i\in\{1, 2, 3\}$,  is $1$.
  \item $n=p_1p_2p_3q$ and  $(\frac{p_i}{\ell})=-1$ for all $i\in\{1, 2, 3\}$.
\end{enumerate}
\end{them}
\begin{exams}Here are some examples.
\begin{enumerate}[\rm1.]
 \item For  $n=p_1p_2=37.89\equiv1\pmod4$ and  $\ell=5\equiv5\pmod8$, we have $(\frac{p_1}{\ell})=-1$, $(\frac{p_2}{\ell})=1$ and $H$ is of type $(2, 4, 4)$.\\
 For  $n=2p_1p_2=2.41.53.89\equiv2\pmod4$ and  $\ell=29\equiv5\pmod8$, we have $(\frac{p_1}{\ell})=-1$, $(\frac{p_2}{\ell})=1$ and $H$ is of type $(2, 2, 4)$.
 \item For  $n=p_1p_2p_3=17.53.89\equiv1\pmod4$ and  $\ell=61\equiv5\pmod8$, we have $(\frac{p_1}{\ell})=-1$, $(\frac{p_2}{\ell})=-1$, $(\frac{p_3}{\ell})=-1$ and $H$ is of type $(2, 2, 2)$.\\
     For  $n=2p_1p_2p_3=2.17.61.89\equiv2\pmod4$ and  $\ell=29\equiv5\pmod8$, we have $(\frac{p_1}{\ell})=-1$, $(\frac{p_2}{\ell})=-1$, $(\frac{p_3}{\ell})=-1$ and $H$ is of type $(2, 2, 2)$.
\item For  $n=2pq=2.53.79\equiv-2\pmod4$ and  $\ell=37\equiv5\pmod8$, we have $(\frac{p}{\ell})=1$ and $H$ is of type $(2, 2, 4)$.
\item For  $n=2q_1q_2=2.59.83\equiv2\pmod4$ and  $\ell=29\equiv5\pmod8$, we have $(\frac{q_1}{\ell})=(\frac{q_2}{\ell})=-1$ and $H$ is of type $(2, 2, 4)$.\\
For  $n=q_1q_2=67.71\equiv1\pmod4$ and  $\ell=37\equiv5\pmod8$, we have $(\frac{q_1}{\ell})=(\frac{q_2}{\ell})=-1$ and $H$ is of type $(2, 2, 2)$.
\item For  $n=q_1q_2q_3=19.47.71\equiv-1\pmod4$ and  $\ell=37\equiv5\pmod8$, we have $(\frac{q_1}{\ell})=-1$, $(\frac{q_2}{\ell})=(\frac{q_3}{\ell})=1$ and $H$ is of type $(2, 2, 2)$.
\item For  $n=2q_1q_2q_3=2.7.67.71\equiv-2\pmod4$ and  $\ell=53\equiv5\pmod8$, we have $(\frac{q_1}{\ell})=1$, $(\frac{q_2}{\ell})=(\frac{q_3}{\ell})=-1$ and $H$ is of type $(2, 2, 2)$.
\item For  $n=p_1p_2q=13.17.43\equiv3\pmod4$ and  $\ell=61\equiv5\pmod8$, we have $(\frac{p_1}{\ell})=-1$, $(\frac{p_2}{\ell})=1$  and $H$ is of type $(2, 4, 4)$.
    \item For  $n=2p_1p_2q=2.29.53.79\equiv2\pmod4$ and  $\ell=61\equiv5\pmod8$, we have $(\frac{p_1}{\ell})=(\frac{p_2}{\ell})=-1$  and $H$ is of type $(2, 2, 2)$.
\item For  $n=pq_1q_2=37.47.71\equiv1\pmod4$ and  $\ell=5\equiv5\pmod8$, we have $(\frac{p}{\ell})=-1$, $(\frac{q_1}{\ell})=-1$, $(\frac{q_2}{\ell})=1$  and $H$ is of type $(2, 2, 2)$.\\
    For  $n=2pq_1q_2=2.17.31.83\equiv2\pmod4$ and  $\ell=37\equiv5\pmod8$, we have $(\frac{p}{\ell})=-1$, $(\frac{q_1}{\ell})=-1$, $(\frac{q_2}{\ell})=1$  and $H$ is of type $(2, 2, 2)$.
\item For  $n=pq_1q_2q_3=5.43.31.71\equiv3\pmod4$ and  $\ell=13\equiv5\pmod8$, we have $(\frac{p}{\ell})=-1$, $(\frac{q_1}{\ell})=1$, $(\frac{q_2}{\ell})=(\frac{q_3}{\ell})=-1$  and $H$ is of type $(2, 2, 2)$.
\item For  $n=p_1p_2p_3q=13.17.29.83\equiv3\pmod4$ and  $\ell=37\equiv5\pmod8$, we have  $(\frac{p_1}{\ell})=(\frac{p_2}{\ell})=(\frac{p_3}{\ell})=-1$  and $H$ is of type $(2, 2, 2)$.
   \end{enumerate}
\end{exams}
\begin{rema}
 If the integer $n$ does not take any value in Theorems  \ref{11}, \ref{8}, \ref{9}, \ref{10}, then $r_2(H)\geq4$.
\end{rema}
\section*{Acknowledgment}
We would like to thank the unknown referee  for his/her several helpful suggestions and for calling our attention
to the missing details.


\small
\begin{thebibliography}{999}
\bibitem{Azizi 04}
A. Aziz.  \emph{Sur le $2$-groupe de classes de certains corps de nombres, }{Ann.  Sci.  Maht.  Québec 28,  no. 1-2, (2004), 37-44.}

\bibitem{AM01} A. Azizi and A. Mouhib, {\it Sur le rang du $2$-groupe de classes de $\mathbb{Q}( \sqrt{m},\sqrt{d})$ où $m=2$ ou un premier $p \equiv 1 \bmod{4}$,}{ Trans. Amer. Math. Soc. 353, No 7 (2001), 2741-2752.}

\bibitem{AM04} A. Azizi and A. Mouhib, {\it Le $2$-rang du groupe de classes de certains corps biquadratiques et applications,}{ Int. J. Math.  vol 15, No. 02, (2004), 169-182.}

\bibitem{Baker 13}
A. Baker, \emph{An introduction to Galois theory, }{Lecture note (2013).}

\bibitem{Ba71}
H. Bauer,  \emph{Zur Berechnung der $2$-Klassenzahl der quadratischen Zahlk\"orper mit genau zwei verschieden
	Diskriminantenprimteilern, }{J.  reine angew.  Math.  248 (1971),  4$2$-46.}

\bibitem{Parry 77}
E. Brown and C. J. Parry, \emph{The $2$-class group of certain biquadratic number fields, }{J.  reine angew.  Math.,   vol  295,  (1977) 61-71.}

\bibitem{BrPa78}
E. Brown and  C. J. Parry, \emph{The $2$-class group of certain biquadratic number fields II, }{Pacific J.  Math.   vol 78,  No.  1, (1978), 61-71.}


\bibitem{Gras-M-N 80-81}
 M. N. Gras, \emph{Table numérique du nombre de classes et des unités des extensions cycliques réelles de degré 4 de $\mathbb{Q}$, }{Publ.  Math.  Fac.  Sciences Bensançon,  Théorie des Nombres (1977-78).}

\bibitem{Gr77}
M. N. Gras,  \emph{Calcul du nombre de classes et des unités des extensions abéliennes réelles de $\QQ$, }{Bull.  Sci.  Math.  101 (1977),   97–129.}

\bibitem{Gr79}
M. N. Gras,  \emph{Classes et unités des extensions cycliques réelles de degré $4$ de $\QQ$, }{Ann.  Inst.  Fourier (Grenoble) 29 (1979),  107–124.}

\bibitem{Hasse 48}
H. Hasse, \emph{Arithmetishe Bestimmung von Grundeinheit und Klassenzahl in zyklischen kubischen und biquadratishen Zahlkorpern, }{Abh.  Deutsche Akad.  Wiss.  Berlin,  Math,  no 2,  (1948) 1-95.}


\bibitem{Ha70}
H. Hasse,  \emph{An algorithm for determining the structure of the $2$-Sylow-subgroup of the divisor class
group of a quadratic number fleld, }{Symposia. Mat.  15 (1975),  341-352.}

\bibitem{wi87}
K. Hardy, R. H. Hudson, D. Richman,  K. S. Williams,  and M. Holtz, \emph{Calculation of the class numbers of imaginary cyclic quartic fileds, }{Journal. Math. comp.   vol 49, (1987), 615-620.}

\bibitem{Hi94}
D. Hilbert,  \emph{Bericht \"uber den algebraischen Zahlk\"orper, }{Jber.  Deutsche Math. Verein.  4, 1894-95.}

\bibitem{Parry 90}
J. A. Hymo,  C. J. Parry, \emph{On relative integral bases for cyclic quartic fields, }{J. Number Theory, vol  34,  Issue 2, (1990), 189-197.}

\bibitem{Leo}
 H. W. Leopoldt,  \emph{Uber einheitengruppe und klassenzahl reeller abelscher zahlk\"orper, }{Abh. Deutsche Akad. Wiss. Berlin. Math. 2, (1953), 1-48.}

\bibitem{McPaRa95}
T. M. McCall, C. J. Parry and R. R. Ranalli, \emph{Imaginary bicyclic biquadratic fields with
cyclic $2$-class group,}{ J. Number. Theor, 53 (1995), 88-99.}

\bibitem{McPaRa97}
 T. M. McCall, C. J. Parry and R. R. Ranalli, \emph{The $2$-rank of the class group of imaginary bicyclic biquadratic fields,}{ Can. J. Math.  vol  49, (1997), 283-300.}

\bibitem{Parry 75}
C. J. Parry,  \emph{Pure quartic number fields whose class numbers are even, }{J.  reine angew.  Math.   vol 264, (1975), 102-112.}

\bibitem{Pa75}
C. J. Parry,  \emph{Real quadratic fields with class number divisible by $5$, }{Math. Comput.  vol  31, (1977),  1019-1029.}

\bibitem{Pa80}
C. J. Parry,  \emph{A genus theory for quartic fields, }{J.  reine. angew.  Math.  314, (1980),  40–71.}

\bibitem{Xianke 84}
Zh. Xianke, \emph{Cyclic quartic fields and genus theory of their subfields, }{J. Number Theory,  no 18, (1984), 350-355.}


\bibitem{Zink 66-67}
O. Zink, \emph{Extension cycliques de degré $2^n$ sur $\mathbb{Q}$, }{Séminaire delange-Pisot-Poitou.  Théorie des nombres,  Tome 8,  no 2,  exp. no 16, (1966-1967),  1-12.}
\end{thebibliography}
\end{document}